\documentclass[a4paper, 11pt]{article}
\usepackage[utf8]{inputenc}
\usepackage[margin=1in]{geometry} 
\usepackage{amsmath,amsthm,amssymb,amsfonts,amscd}
\usepackage{mathtools}
\usepackage{enumitem}
\usepackage{indentfirst}
\usepackage{thmtools}
\usepackage{thm-restate}
\usepackage[numbers]{natbib}
\usepackage[affil-it]{authblk}
\setlist[enumerate,1]{label={(\alph*)}}
\usepackage{etoolbox}
\makeatletter
\pretocmd\@maketitle
  {\def\@makefnmark{\hbox{\@textsuperscript{\normalfont\@thefnmark}}}}
  {}{\FAIL}
\makeatother

\title{Eigenvalue Paths Arising From Matrix Paths}
\author{
    Eric Jankowski\,\textsuperscript{a,}\thanks{Corresponding author. Postal address: 1935 Addison St \#214, Berkeley, CA 94704}  and Charles R.\ Johnson\,\textsuperscript{b} \\
    \textsuperscript{a,b}\,Department of Mathematics \\
    College of William \& Mary \\
    P.O.\ Box 8795 \\
    Williamsburg, VA 23187}
\date{\today}

\newcommand{\N}{\mathbb{N}}

\newcommand{\R}{\mathbb{R}}
\newcommand{\C}{\mathbb{C}}
\newcommand{\CC}{\mathcal{C}}
\newcommand{\PP}{\mathcal{P}}
\newcommand{\SSS}{\mathcal{S}}
\newcommand{\HH}{\mathcal{H}}
\newcommand{\OO}{\mathcal{O}}

\newcommand{\spec}{\Sigma}

\DeclareMathOperator{\mult}{mult}
\DeclareMathOperator{\Tr}{Tr}

\theoremstyle{plain}
\newtheorem{Thm}{Theorem}
\newtheorem{Lemma}{Lemma}[section]
\newtheorem{Cor}[Lemma]{Corollary}
\newtheorem{Prop}[Lemma]{Proposition}
\theoremstyle{definition}
\newtheorem{Def}[Lemma]{Definition}

\newtheorem{Rmk}[Lemma]{Remark}

\newenvironment{usethmcounterof}[1]{%
  \renewcommand\theThm{\ref{#1}}\Thm}{\endThm\addtocounter{Thm}{-1}}
  
\newcommand\blfootnote[1]{%
  \begingroup
  \renewcommand\thefootnote{}\footnote{#1}%
  \addtocounter{footnote}{-1}%
  \endgroup
}

\begin{document}
\maketitle

\blfootnote{E-mail addresses: ejankowsi@berkeley.edu (E.\ Jankowski), crjohn@wm.edu (C.\ R.\ Johnson)}
\blfootnote{Declarations of interest: none.}
\blfootnote{This work was supported by the 2019 National Science Foundation grant DMS \#1757603.}

\begin{abstract}
It is known (see e.g.\ \cite{Bounds}, \cite{RootContinuity}, \cite{MAnalysis}, \cite{PertTheoryKato}) that continuous variations in the entries of a complex square matrix induce continuous variations in its eigenvalues. If such a variation arises from one real parameter $\alpha \in [0,1]$, then the eigenvalues follow continuous paths in the complex plane as $\alpha$ shifts from $0$ to $1$. The intent here is to study the nature of these \emph{eigenpaths}, including their behavior under small perturbations of the matrix variations, as well as the resulting \emph{eigenpairings} of the matrices that occur at $\alpha = 0$ and $\alpha=1$. We also give analogs of our results in the setting of monic polynomials.
\end{abstract}




\section{Introduction}\label{SectionIntro}
We are interested in continuous paths through the space $M_n$ of $n$-by-$n$ complex matrices and the variations in eigenvalues along these paths. To this end, a \emph{matrix path} will refer to a continuous function $\CC : [0,1] \to M_n$, where we use $\CC_\alpha$ to represent its value at $\alpha \in [0,1]$. Unless otherwise stated, we assume $A=\CC_0$ and $B=\CC_1$. As proven in Theorem 2.5.2 of \cite{PertTheoryKato}, for such a path $\CC$ through $M_n$, there exists a continuous parameterization of the eigenvalues along this path. That is, there are $n$ continuous functions $\gamma_1, ..., \gamma_n : [0,1] \to \C$ such that $\spec(\CC_\alpha) = \{\gamma_1(\alpha), ..., \gamma_n(\alpha)\}$ for all $\alpha \in [0,1]$, where we use $\spec(A)$ to denote the size-$n$ multiset spectrum of a matrix $A \in M_n$.

It can be the case that the parameterization $\gamma_1, ..., \gamma_n$ is not unique (e.g.\ if two distinct paths $\gamma_i, \gamma_j$ intersect at some point but differ on either side of that point). This case is of particular interest, since small perturbations of the matrix path $\CC$ often ``break apart" intersecting paths, thereby removing these points of intersection. We analyze perturbations of this type with the aim of showing that any parameterization of the perturbed spectrum is close to some parameterization $\gamma_1, ..., \gamma_n$ of the initial spectrum $\spec(\CC_\alpha)$, and further that there is no ``canonical" parameterization of $\spec(\CC_\alpha)$.

Of special interest is the convex case, in which the matrix path in consideration is given by $(1-\alpha)A + \alpha B$ for $A, B \in M_n$. This matrix path has many notable applications to quantum physics (see e.g.\ \cite{Physics1}, \cite{Physics2}), motivating some of our discussion of this subject.

We begin with the following definitions that guide our analysis.

\begin{Def}
Given a matrix path $\CC$, choose $n$ continuous functions $\gamma_1, ..., \gamma_n : [0, 1] \to \C$, possibly not all distinct, such that the multiset equality
\begin{align}
\spec(\CC_\alpha) = \{\gamma_1(\alpha), ..., \gamma_n(\alpha)\} \label{EigenPSDef}
\end{align}
holds for all $\alpha \in [0,1]$. The multiset $\{\gamma_1, ..., \gamma_n\}$ of paths is called a \emph{$\CC$-eigenpath set}, and each element is a \emph{$\CC$-eigenpath}.
\end{Def}

\begin{Def}
Given a $\CC$-eigenpath set $\{\gamma_1, ..., \gamma_n\}$, we may define a bijection $p : \spec(A) \to \spec(B)$ by $\gamma_j(0) \mapsto \gamma_j(1)$ for $j=1,...,n$. We call a bijection induced by these paths a \emph{$\CC$-eigenpairing} of the eigenvalues of $A$ and $B$. If there is exactly one $\CC$-eigenpairing, we say that it is \emph{unambiguous}.
\end{Def}

\begin{Rmk}
Notice that if $A$ (or $B$) has a repeated eigenvalue $\lambda$, then there cannot be an unambiguous eigenpairing due to the distinction we make between the first and second occurrences of $\lambda$ in the multiset $\spec(A)$.
\end{Rmk}

\begin{Def}
The \emph{$\CC$-eigenregion}, denoted by $E_\CC$, is the set of all eigenvalues realized by the matrices $\CC_\alpha$, each adorned with the parameter $\alpha \in [0,1]$ of the corresponding matrix. That is, we may write
\begin{align}
    E_\CC = \bigcup_{\alpha \in [0, 1]} \{(\lambda, \alpha) \mid \lambda \in \sigma(\CC_\alpha)\}
\end{align}
so $E_\CC \subseteq \C \times [0, 1]$ with the relative topology inherited from the standard product topology.
\end{Def}

\begin{Def}
We say that the point $(\lambda, \alpha) \in E_\CC$ is an \emph{ambiguity} if $\lambda$ is a repeated eigenvalue of $C_{\alpha}$. If $\mult(\lambda, \alpha)$ denotes the multiplicity of $\lambda$ in $\spec(\CC_\alpha)$, we say that an ambiguity $(\lambda_0, \alpha_0) \in E_\CC$ is \emph{singular} if for all open neighborhoods $\OO \subseteq E_\CC$ of $(\lambda_0, \alpha_0)$, there is a $(\lambda, \alpha) \in \OO$ such that $\mult(\lambda, \alpha) < \mult(\lambda_0, \alpha_0)$. Otherwise we say it is \emph{nonsingular}.
\end{Def}

\begin{Def}
We write $C_\alpha = (1-\alpha)A + \alpha B$ to denote the convex path from $A$ to $B$. Further, we will often use the prefix \emph{convex} (e.g.\ convex eigenpath set, convex eigenregion, etc.) when referring to objects induced by this matrix path.
\end{Def}

The overall structure of this paper will be as follows: In sections \ref{SectionBasicFacts} and \ref{SectionInvariants}, we present our basic results on eigenpaths, eigenpairings, and ambiguities. Many of these results will motivate definitions and aid us with later proofs. Section \ref{SectionAnalytic} contains a brief analysis of some crucial ideas from analytic perturbation theory that will later be applied to prove a theorem on non-analytic matrix paths (Theorem \ref{RippingThm}). The next four sections will be dedicated to proving the four theorems below.

Theorems \ref{ClosePathsThm} and \ref{RippingThm} characterize achievable eigenpath sets for matrix paths that are norm-close to a given matrix path $\CC$. In particular, we see that eigenpaths are rather well-behaved under small perturbations of $\CC$. These first two results are especially practical when dealing with matrix paths that exhibit undesirable behavior at infinitely many points or intervals, as evidenced by the application of Theorem \ref{ClosePathsThm} to our proof of Theorem \ref{ConvexReductionThm}.

In these theorems, as well as in the rest of this paper, we use the generalized matrix norm $\|A\| = \max_{i,j} |a_{i,j}|$. Recall that this norm is equivalent to all other (generalized) matrix norms; our choice is purely for convenience.

\begin{restatable}{Thm}{ClosePathsThm}\label{ClosePathsThm}
Let $\varepsilon > 0$ and let $\CC$ be a matrix path. Then there is a $\delta > 0$ such that for any matrix path $\CC'$ with $\|\CC_\alpha - \CC_\alpha'\| < \delta$ for all $\alpha \in [0,1]$ and any $\CC'$-eigenpath set $\{\gamma_1', ..., \gamma_n'\}$, there is a $\CC$-eigenpath set $\{\gamma_1, ..., \gamma_n\}$ satisfying $|\gamma_j(\alpha) - \gamma_j'(\alpha)| < \varepsilon$ for all $\alpha \in [0,1]$.
\end{restatable}

\begin{restatable}{Thm}{RippingThm}\label{RippingThm}
Let $\varepsilon > 0$ and let $\CC$ be a matrix path with $\CC$-eigenpath set $\{\gamma_1, ..., \gamma_n\}$. Then there is a matrix path $\CC'$ admitting a unique $\CC'$-eigenpath set $\{\gamma_1', ..., \gamma_n'\}$ such that $\|\CC_\alpha - \CC_\alpha'\| < \varepsilon$ and $|\gamma_j(\alpha) - \gamma_j'(\alpha)| < \varepsilon$ for all $\alpha \in [0,1]$ and $j=1,...,n$.
\end{restatable}

Theorem $\ref{ConvexReductionThm}$ gives a condition under which we can expect eigenpairings of a particular type of matrix path to coincide with convex eigenpairings. Here we write $f \lor g$ to denote the least upper bound (i.e.\ pointwise maximum) of the functions $f,g : [0,1] \to \R$. This reduction is particularly useful in view of Theorem \ref{2x2Theorem}, which completely determines the convex eigenpairings in the $2$-by-$2$ case.

\begin{restatable}{Thm}{ConvexReductionThm}\label{ConvexReductionThm}
Let $f$ and $g$ be continuous functions $[0,1] \to \R$ satisfying $f(0) = g(1) = 1$ and $f(1) = g(0) = 0$ so that $\CC_\alpha = f(\alpha)A+g(\alpha)B$ is a matrix path from $A$ to $B$. If $(f \lor g)(\alpha) \geq 0$ for all $\alpha \in [0,1]$, then any convex eigenpairing $p$ is also a $\CC$-eigenpairing.
\end{restatable}

\begin{restatable}{Thm}{2x2Theorem}\label{2x2Theorem}
Suppose that $A \in M_2$ has distinct eigenvalues $\lambda_1, \lambda_2$, and further that $B \in M_2$ has distinct eigenvalues $\mu_1,\mu_2$. Then the convex eigenpairings of $A$ and $B$ are determined entirely by the proximity of their eigenvectors and the quantity $\frac{\mu_1-\mu_2}{\lambda_1-\lambda_2}$.
\end{restatable}

In the final section, we give analogs of Theorems 1-3 for paths of polynomials and their corresponding paths of roots.


\section{Some Basic Facts}\label{SectionBasicFacts}

First we give some motivation to the definition of an ambiguity.

\begin{Prop}\label{AmbiguityProp}
The following are equivalent:
\begin{enumerate}
    \item There is an unambiguous $\CC$-eigenpairing.
    
    \item $\CC_\alpha$ has $n$ distinct eigenvalues for all $\alpha \in [0, 1]$.
    
    \item $E_\CC$ contains no ambiguities.
    
    \item $E_\CC$ consists of $n$ connected components.
\end{enumerate}
\end{Prop}

\begin{proof}
We will first show that (a)$\implies$(b). Suppose that (b) does not hold, so some $\CC_\alpha$ has fewer than $n$ distinct eigenvalues. Notice that if $\alpha = 0$, then we may trivially find multiple pairings by switching which ``copy" of the repeated eigenvalue $\lambda \in \spec(A)$ we map to which value in $\spec(B)$. The same is true if $\alpha = 1$. We may therefore assume that both $A$ and $B$ have $n$ distinct eigenvalues. Then we may find two associated paths, say $\gamma_1$ and $\gamma_2$, such that $\gamma_1(\alpha) = \gamma_2(\alpha)$. We may replace these paths with $\gamma_1'$ and $\gamma_2'$ such that $\gamma_1'(t) = \gamma_1(t)$ for $t \leq \alpha$ and $\gamma_1'(t) = \gamma_2(t)$ for $t \geq \alpha$, and the opposite for $\gamma_2'$. This new collection of paths induces a new pairing mapping $\gamma_1(0) \mapsto \gamma_2(1)$ and $\gamma_2(0) \mapsto \gamma_1(1)$. Therefore there are multiple pairings in this case, so indeed (a)$\implies$(b).

By definition of ambiguities, we also have the implication (b)$\implies$(c). Now suppose that (d) holds, so the $j$th connected component may be described by a continuous path $\gamma_j$. Then the collection $\gamma_1, ..., \gamma_n$ of paths induces a single pairing $p$. Since this is the only choice of paths, we conclude that (d)$\implies$(a).

We are now reduced to showing that (c)$\implies$(d). Suppose that (d) does not hold, so there are some paths, say $\gamma_1$ and $\gamma_2$, that intersect at the same $\alpha$. Then $\CC_\alpha$ has a repeated eigenvalue of $\gamma_1(\alpha) = \gamma_2(\alpha)$, so we are finished.
\end{proof}

We will now demonstrate that the singular ambiguities of $E_\CC$ are compact, a fact that will be useful in Lemma \ref{FiniteCoverLem} and that proves a property (Corollary \ref{NSACor}) about points in $E_\CC$ that are not singular ambiguities. We begin with compactness of $E_\CC$ itself:

\begin{Lemma}\label{CompactERLem}
The eigenregion $E_\CC$ is a compact subset of $\C \times [0,1]$.
\end{Lemma}

\begin{proof}
Let $\{\gamma_1, ..., \gamma_n\}$, so $E_\CC$ is the union of the graphs of the $\gamma_j$. Since the $\gamma_j$ are continuous functions into the Hausdorff space $\C$, it follows that each graph is closed. Finally, since the domains of the $\gamma_j$ are compact, we find that each graph is bounded. Therefore $E_\CC$ is compact, as desired.
\end{proof}

\begin{Lemma}\label{CompactSALem}
The singular ambiguities of $E_\CC$ form a compact subset of $E_\CC$.
\end{Lemma}

\begin{proof}
Let $Z$ be the set of singular ambiguities in $E_\CC$. By Lemma \ref{CompactERLem}, it suffices to show that $Z$ is closed. Suppose we have a sequence $\{(\lambda_k, \alpha_k)\}$ of elements of $Z$ that converges to some point $(\lambda, \alpha) \in \C \times [0,1]$. Let $m$ be the maximum integer such that $m=\mult(\lambda_k, \alpha_k)$ for infinitely many $k$. Then there must be a subset $\{\gamma_1, ..., \gamma_m\}$ of a $\CC$-eigenpath set that all intersect at infinitely many of these points. In particular, we may find a subsequence (which we also label $\{(\lambda_k, \alpha_k)\}$) satisfying $\gamma_j(\alpha_k) = \lambda_k$ for $k \in \N$ and $j=1,...,m$. Then by continuity of the $\gamma_j$, we also obtain $\gamma_j(\alpha) = \lambda$.

Now consider an open neighborhood $\OO$ of $(\lambda, \alpha)$. By our assumption on $m$, we know there is a $(\lambda_k, \alpha_k) \in \OO$ such that $\mult(\lambda_k, \alpha_k) = m$. Since $\OO$ is an open neighborhood of the singular ambiguity $(\lambda_k, \alpha_k)$, there must be some other point in $\OO$ with multiplicity less than $m$. Therefore $(\lambda, \alpha)$ is a singular ambiguity, so we are finished.
\end{proof}

\begin{Cor}\label{NSACor}
Suppose that $(\lambda_0, \alpha_0) \in E_\CC$ is not a singular ambiguity. Then there is an open neighborhood $\OO \subseteq E_\CC$ of $(\lambda_0, \alpha_0)$ that contains no singular ambiguities.
\end{Cor}

\begin{proof}
If not, then there would be a sequence of singular ambiguities converging to $(\lambda_0, \alpha_0)$, in which case it too would be a singular ambiguity.
\end{proof}

The following proposition provides useful information about how we may construct a $\CC$-eigenpath set. In particular, we may simply choose one eigenpath at a time until property \eqref{EigenPSDef} is satisfied. Additionally, it confirms that any continuous function $\gamma : [0,1] \to \C$ such that $\gamma(\alpha) \in \spec(\CC_\alpha)$ for all $\alpha \in [0,1]$ is indeed a $\CC$-eigenpath.

\begin{Prop}\label{PathExtProp}
For $k \leq n$, suppose there are $k$ continuous functions $\gamma_1, ..., \gamma_k : [0,1] \to \C$ such that
\begin{align}
\spec(\CC_\alpha) \supseteq \{\gamma_1(\alpha), ..., \gamma_k(\alpha)\} \label{PEigenPSDef}
\end{align}
as a multiset for all $\alpha \in [0,1]$. Then there are continuous functions $\gamma_{k+1}, ..., \gamma_n : [0,1] \to \C$ such that $\{\gamma_1, ..., \gamma_n\}$ is a $\CC$-eigenpath set.
\end{Prop}

\begin{proof}
By induction, it suffices to show for $k<n$ that there is a path $\gamma_{k+1}$ that preserves property \eqref{PEigenPSDef} when it is added to the union.

Suppose towards a contradiction that such a choice of $\gamma_{k+1}$ were not possible, so any such choice of a function $\gamma_{k+1}$ must be discontinuous. Since there are only finitely many paths in the set $\{\gamma_1, ..., \gamma_k\}$, we know that $\gamma_{k+1}$ may be chosen so that it is discontinuous at finitely many points. Let $\alpha_0 > 0$ be the least point at which $\gamma_{k+1}$ must be discontinuous, and suppose that $m$ of the paths $\gamma_j$ for $j=1, ..., k$ satisfy $\gamma_j(\alpha_0) = \lim_{\alpha \uparrow \alpha_0} \gamma_{k+1}(\alpha)$. Then any choice $\gamma_1',...\gamma_n'$ of $\CC$-eigenpaths must have $m+1$ that approach $\lim_{\alpha \uparrow \alpha_0} \gamma_{k+1}(\alpha)$ from the left at $\alpha_0$ and only $m$ that approach it from the right at $\alpha_0$. Therefore there can be no $\CC$-eigenpath set, a contradiction.
\end{proof}

By using the initial set $\{\gamma_1, ..., \gamma_k\}$ to determine a partial function $p : \spec(A) \to \spec(B)$ given by $p(\gamma_j(0)) = \gamma_j(1)$ for $j=1,...,k$, we can then use the extension from the previous lemma to extend $p$ to an $\CC$-eigenpairing:

\begin{Cor}\label{PairingExtCor} Suppose that $p$ is a partial bijection $\spec(A) \to \spec(B)$ induced by the continuous functions $\gamma_1, ..., \gamma_k$ satisfying \eqref{PEigenPSDef}. Then there is a $\CC$-eigenpairing that extends $p$. 
\end{Cor}

\section{Invariants}\label{SectionInvariants}

There are several key invariants in the problem of finding the $\CC$-eigenregion and $\CC$-eigenpairings of $A$ and $B$. To simplify the proofs of our main results, we will first analyze these invariants. The first few results in this section allow us to modify the matrix path $\CC_\alpha$ in a consistent way and expect similar eigenpairings to occur.

\begin{Lemma}[Uniform Similarity]
Let $S \in M_n$ be nonsingular, and let $\CC^S$ be the matrix path given by $\alpha \mapsto S \CC_\alpha S^{-1}$. Then $E_{\CC^S} = E_\CC$.
\end{Lemma}

\begin{proof}
Similar matrices share the same spectrum, so the spectra at each $\alpha$ coincide.
\end{proof}

Since the eigenregions in the above proposition are identical, the $\CC$-eigenpairings coincide with the $\CC^S$-eigenpairings. In the following lemma, we will see that the same is essentially true, with the caveat that each point must be scaled and shifted by some $a,b \in \C$.

\begin{Lemma}[Uniform Scaling and Shifting]\label{ScaleShiftInvLem}
Let $a,b \in \C$, and let $a\CC+bI$ be the matrix path given by $\alpha \mapsto a\CC_\alpha + bI$. Then
$$E_{a\CC+bI} = \big\{ (a\lambda+b, \alpha) \mid (\lambda, \alpha) \in E_\CC \big\}$$
\end{Lemma}

\begin{proof}
Each eigenvalue of $a\CC_\alpha+bI$ is $a\lambda+b$ for some $\lambda \in \spec(\CC_\alpha)$.
\end{proof}

In the convex case with matrix path $C_\alpha = (1-\alpha)A + \alpha B$, this lemma admits a slight generalization:

\begin{Lemma}[Convex Scaling and Shifting]\label{PosScalLem}
Let $c > 0$ and $d \in \C$, and denote by $C'$ the convex matrix path from $A$ to $cB+dI$. Then
$$E_{C'} = \Big\{ \big((1-\alpha + \alpha c)\lambda + \beta(\alpha) d, \alpha \big) \mid (\lambda, \beta(\alpha)) \in E_C \Big\}$$
where $\beta : [0, 1] \to [0, 1]$ is the strictly increasing smooth bijection $\alpha \mapsto \frac{\alpha c}{1-\alpha + \alpha c}$.
\end{Lemma}

\begin{proof}
Note that 
\begin{align*}
    \sigma\big((1-\alpha)A + \alpha cB\big) &= (1-\alpha + \alpha c) \sigma \left( \frac{1-\alpha}{1-\alpha + \alpha c}A + \frac{\alpha c}{1-\alpha + \alpha c} B \right) \\
    &= (1-\alpha + \alpha c) \sigma( C_{\beta(\alpha)})
\end{align*}
and
\begin{align*}
    \sigma\big((1-\alpha)A + \alpha(B+dI)\big) = \sigma(C_\alpha + \alpha d I),
\end{align*}
so composing these two operations yields the result.
\end{proof}

\begin{Rmk}
In particular, if $p$ is a convex eigenpairing of $A$ and $B$, then the map $\lambda \mapsto c p(\lambda) + d$ is a convex eigenpairing of $A$ and $cB+dI$. We remark that multiplying one matrix by a negative or non-real scalar without also scaling the other does not preserve pairings in a predictable way. For a concrete realization of this phenomenon, see section \ref{Section2x2} and note that scaling just one matrix by an element of $\C \backslash \R^+$ will change the crucial quantity $\arg(\frac{\mu}{\lambda})$.
\end{Rmk}

Moreover, we may modify the scope of our path by inverting it, truncating it, or extending it as follows. In remainder of this section, we will assume that $p$ is a $\CC$-eigenpairing.

\begin{Lemma}[Inversion]\label{InvertibilityLem}
Denote by $\CC^R$ the path $\CC$ with reversed orientation, so $\CC^R$ is a matrix path from $B$ to $A$. Then the inverse map $p^{-1}$ is a $\CC^R$-eigenpairing of $B$ and $A$.
\end{Lemma}

\begin{proof}
    By reversing the orientations of the eigenpaths $\gamma_j$ that induce the eigenpairing $p : \spec(A) \to \spec(B)$, we obtain paths that induce the eigenpairing $p^{-1} : \spec(B) \to \spec(A)$.
\end{proof}

\begin{Lemma}[Truncation]
Let $0\leq\alpha<\beta\leq1$, and suppose that $\{\gamma_1, ..., \gamma_n\}$ is a $\CC$-eigenpath set. Denote by $\gamma_j'$ the restriction of the path $\gamma_j$ to $[\alpha, \beta]$, so each $\gamma_j'$ is a continuous function $[\alpha, \beta] \to \C$. If $\CC'$ is the restriction of $\CC$ to $[\alpha, \beta]$, then $\{\gamma_1', ..., \gamma_n'\}$ is a $\CC'$-eigenpath set.
\end{Lemma}

\begin{proof}
The paths $\gamma_j'$ satisfy the same property \eqref{EigenPSDef} that characterizes eigenpath sets.
\end{proof}

Since truncation of matrix paths truncates eigenpaths correspondingly, we may abuse notation and say that $\{\gamma_1', ..., \gamma_n'\}$ is a $\CC$-eigenpath set from $\CC_\alpha$ to $\CC_\beta$ when we actually mean that it is an eigenpath set for the truncation of $\CC$ to $[\alpha, \beta]$. Similarly, we may refer to a $\CC$-eigenpairing and the $\CC$-eigenregion of $\CC_\alpha$ and $\CC_\beta$.

\begin{Lemma}[Concatenation]\label{Concatenation}
Suppose that $p$ is a $\CC$-eigenpairing for $A$ and $A'$, and further that $p'$ is a $\CC'$-eigenpairing for $A'$ and $A''$. Define $\CC''$ as the concatenation of the paths $\CC$ and $\CC'$. Then the composition $p' \circ p$ is a $\CC''$-eigenpairing for $A$ and $A''$.
\end{Lemma}

\begin{proof}
Let $\{\gamma_1, ..., \gamma_n\}$ be a $\CC$-eigenpath set corresponding to $p$ and $\{\gamma_1', ..., \gamma_n'\}$ a $\CC'$-eigenpath set corresponding to $p'$. Re-index these paths so that $\gamma_j(1) = \gamma_j'(0)$. Define $\gamma_j''$ as the concatenation of the paths $\gamma_j$ and $\gamma_j'$. Then $\{\gamma_1'', ..., \gamma_n''\}$ is a $\CC''$-eigenpath set that induces the pairing $p' \circ p$, finishing our proof.
\end{proof}

\begin{Cor}\label{AllBijectionsCor}
If $\CC_{\alpha_0} = cI$ for some $\alpha_0 \in [0,1]$ and $c \in \C$, then every bijection $\spec(A) \to \spec(B)$ is a $\CC$-eigenpairing.
\end{Cor}

\begin{proof}
Apply the previous proposition, using $A' = cI$ and $A''=B$.
\end{proof}

\begin{Lemma}[Combination]\label{Combination}
For $j=1,...,m$ let $\CC^{(j)} : [0,1] \to M_{n_j}$ be continuous, and suppose that $\CC_\alpha$ is block upper (or lower) triangular with blocks $\CC^{(1)}_\alpha, ..., \CC^{(m)}_\alpha$. Then $E_\CC = \bigcup_{j=1}^m E_{\CC^{(j)}}$, and so any collection $\{p_j\}_{j=1}^m$ such that $p_j$ is a $\CC^{(j)}$-eigenpairing induces a $\CC$-eigenpairing $p$ given by $\lambda \mapsto p_j(\lambda)$ when $\lambda \in \spec(\CC^{(j)})$.
\end{Lemma}

\begin{proof}
The characteristic polynomial of $\CC_\alpha$ is the product of those of the $\CC^{(j)}_\alpha$, so the $\CC$-eigenregion is the union of the $\CC^{(j)}$-eigenregions. It follows that if $\{\gamma_{j,1}, ..., \gamma_{j, n_j}\}$ is a $\CC^{(j)}$-eigenpath set that induces the eigenpairing $p_j$, then the set $\{\gamma_{j, i} \mid 1 \leq j \leq m, 1 \leq i \leq n_j\}$ is a $\CC$-eigenpath set for $A$ and $B$. Furthermore, this eigenpath set induces the $\CC$-eigenpairing $p$.
\end{proof}

\begin{Rmk}
The converse of the above lemma does not hold in the sense that not all $\CC$-eigenpairings directly result from $\CC^{(j)}$-eigenpairings. In particular, this event will occur if some $E_{\CC^{(i)}}$ and $E_{\CC^{(j)}}$ have nonempty intersection for $i \neq j$.
\end{Rmk}


\section{Analytic Matrix Paths}\label{SectionAnalytic}

In this section we enforce the condition that the entries of $\CC_\alpha$ be analytic in $\alpha$. This case has been treated rather thoroughly in \cite{PertTheoryKato} (wherein singular ambiguities are called ``exceptional points") and \cite{PertTheoryBaum}, so we will simply present some main results through the lens of eigenpairings. Principally, we aim to show that $\CC$-eigenpaths are analytic except at finitely many singular ambiguities, and that a $\CC$-eigenpath set may be chosen so that any two eigenpaths either coincide entirely or intersect at just finitely many points. Though interesting in their own right, these results mainly serve to aid with our later proof of Theorem \ref{RippingThm}.

We will first cite some intermediate lemmas, noting that the characteristic polynomial of $\CC_\alpha$ is analytic on $[0,1]$ and therefore holomorphic on a domain in $\C$ containing $[0,1]$.

\begin{Lemma}\label{AnalyticityInts}
Denote by $\HH$ the space $\C$-valued functions of $\alpha$ that are holomorphic on a domain containing the real interval $[0,1]$.
\begin{enumerate}
    \item Let $\chi(\alpha,t) \in \HH[t]$ be a monic polynomial in the variable $t$ whose coefficients are holomorphic functions of $\alpha$. Then there exists a unique decomposition
    \begin{align*}
        \chi = \prod_{j=1}^r q_j^{m_j}
    \end{align*}
    of $\chi$ into monic irreducible factors $q_1, ..., q_r \in \HH[t]$. (\cite{PertTheoryBaum}, Corollary 3.2.1.1)
    
    \item Let $q \in \HH[t]$ be monic and irreducible. Then the points $\alpha$ at which $q$ has a multiple root are isolated. (\cite{PertTheoryBaum}, Corollary 3.2.2.2)
    
    \item Let $q_1, q_2 \in \HH[t]$ be relatively prime. Then the points $\alpha$ at which $q_1$ and $q_2$ have a common root are isolated. (\cite{PertTheoryBaum}, Theorem A3.1.1)
    
    \item The simple roots of a polynomial are smooth functions of its coefficients. (\cite{SimpleSmooth})
\end{enumerate}
\end{Lemma}

\begin{Prop}\label{AnalyticityThm}
Every $\CC$-eigenpath set is piecewise-smooth, and the only points at which an eigenpath might not be smooth are the singular ambiguities in $E$.
\end{Prop}

\begin{proof}
Using Lemma \ref{AnalyticityInts}(a), decompose the characteristic polynomial of $\CC_\alpha$ as $\chi = \prod_{j=1}^r q_j^{m_j}$. Suppose that $(\lambda_0, \alpha_0) \in E_\CC$ is not a singular ambiguity. Then by Corollary \ref{NSACor}, there is an open neighborhood $\OO \subseteq E$ of $(\lambda_0, \alpha_0)$ on which points have constant multiplicity $m = \mult(\lambda_0, \alpha_0)$. Suppose without loss of generality that $U$ is connected, so $U = \{(\gamma(\alpha), \alpha) \mid \alpha_1 < \alpha < \alpha_2\}$ for some $\alpha_1 < \alpha_0 < \alpha_2$ and continuous function $\gamma : (\alpha_1, \alpha_2) \to \C$. Since $(\lambda_0, \alpha_0)$ was arbitrary, it suffices to show that $\gamma$ is smooth at $\alpha_0$.

We know by Lemma \ref{AnalyticityInts}(b) that there is some $j$ such that $q_j(\alpha, \gamma(\alpha)) = 0$ for $\alpha_1 < \alpha < \alpha_2$, and that $q_j$ has no multiple roots in this region. Then by (c) we know that $\gamma$ is smooth in the coefficients of $q_j$. Since these coefficients are smooth in $\alpha$ by (d), it follows that $\gamma$ is in fact smooth at $\alpha_0$, as desired.
\end{proof}

\begin{Lemma}\label{PathMultiplicityLem}
There is a $\CC$-eigenpath set $\{\gamma_1, ..., \gamma_n\}$ such that for any $i \neq j$, either $\gamma_i = \gamma_j$ on $[0,1]$ or $\gamma_i$ and $\gamma_j$ agree at finitely many points.
\end{Lemma}

\begin{proof}
Again decompose the characteristic polynomial $\chi(\alpha,t)$ of $\CC_\alpha$ as $\chi = \prod_{j=1}^r q_j^{m_j}$. Notice that the root set of $\chi$ is the union of the root sets of these irreducible factors.

In particular, the eigenregion is $\{(\lambda, \alpha) \mid \alpha \in [0,1], \; \chi_0(\alpha, \lambda) = 0\}$ for $\chi_0 = \prod_{j=1}^r q_j$. By parts (b) and (c) of Lemma \ref{AnalyticityInts}, the points $\alpha$ at which $\chi_0$ has a multiple root are isolated. Since $\alpha \in [0,1]$, there are in fact finitely many such points. We may therefore find a set $\{\gamma_1, ..., \gamma_s\}$ of paths such that each $q_j$ has a root set $\{\gamma_{j_1}(\alpha), ..., \gamma_{j_{n_j}}(\alpha) \mid \alpha \in [0,1]\}$, the $\gamma_i$ intersect at only finitely many points, and the sets $\{\gamma_{j_1}, ..., \gamma_{j_{n_j}}\}$ partition $\{\gamma_1, ..., \gamma_s\}$. Then the multiset of paths wherein each $\gamma_{j_i}$ occurs $m_j$ times is a $\CC$-eigenpath set, so we are finished.
\end{proof}

\begin{Cor}\label{SingularAmbiguityBound}
There are finitely many singular ambiguities in $E_\CC$.
\end{Cor}

\begin{proof}
The singular ambiguities occur precisely at the (finitely many) points of intersection as in Lemma \ref{PathMultiplicityLem}.
\end{proof}


\section{Proof of Theorem 1}

Recall Theorem \ref{ClosePathsThm}, which states that sufficiently small perturbations of a matrix path $\CC$ induce small perturbations of eigenpaths, regardless of the ambiguities in $E_\CC$.

Our strategy for proving this result will be as follows: First, we construct a sufficiently well-behaved finite open cover of the singular ambiguities of $E_\CC$. We then find $\delta>0$ based on certain numerical properties of this cover. Finally, the conditions for a ``well-behaved" open cover will allow us to construct the desired $\CC$-eigenpath set. The following definitions and technical lemmas serve to break up this proof into smaller components, some of which will also be used in our proof of Theorem \ref{RippingThm}.

\begin{Def}
The \emph{diameter} of $X \subseteq E_\CC$ is $\sup \{ |\lambda_1 - \lambda_2| \mid (\lambda_1, \alpha), (\lambda_2, \alpha) \in X \}$.
\end{Def}

\begin{Def}
We say that a $\CC$-eigenpath $\gamma$ \emph{passes through} a subset $X \subseteq E_\CC$ if there is an $\alpha \in [0,1]$ such that $(\gamma(\alpha), \alpha) \in X$.
\end{Def}

\begin{Lemma}\label{SeparationLem}
Let $(\lambda_0, \alpha_0) \in E_\CC$ be a singular ambiguity and $\varepsilon > 0$. Then there is a connected open neighborhood $\OO \subseteq E_\CC$ of $(\lambda_0, \alpha_0)$ with closure $\overline{\OO }$ such that
\begin{enumerate}
    \item The diameter of $\OO $ is less than $\varepsilon$.
    
    \item For any $\CC$-eigenpath $\gamma$ that passes through $\overline{\OO }$, we have $\gamma(\alpha_0) = \lambda_0$.
    
    \item The boundary $\partial \OO $ of $\OO $ contains finitely many points, none of which are singular ambiguities.
\end{enumerate}
\end{Lemma}

\begin{proof}
Fix some $\CC$-eigenpath set $\{\gamma_1, ..., \gamma_n\}$, and find the minimum distance $d>0$ from $\lambda_0$ to $\gamma_j(\alpha_0)$ for $\gamma_j$ such that $\gamma_j(\alpha_0) \neq \lambda_0$. By continuity of the $\gamma_j$, we know there is an open interval $U' \subseteq [0,1]$ containing $\alpha$ such that $|\gamma_j(\alpha) - \gamma_j(\alpha_0)| < \frac{d}{2}, \frac{\varepsilon}{2}$ for all $\alpha \in U'$ and $j=1,...,n$. Now let $\OO ' = \{(\gamma_j(\alpha), \alpha) \mid \alpha \in U', \; \gamma_j(\alpha_0) = \lambda_0\}$. None of the paths that hit $\lambda_0$ at $\alpha_0$ will intersect any other path at any point in $U'$, and every eigenvalue $\lambda$ such that some $(\lambda, \alpha)$ lies in $\OO '$ is less than $\frac{\varepsilon}{2}$ away from $\lambda_0$. Therefore the open set $\OO '$ satisfies properties (a) and (b).

Now for each boundary point $(\lambda, \alpha) \in \partial \OO '$ that is a singular ambiguity, find a connected open neighborhood $V_{\lambda, \alpha}$ of $(\lambda, \alpha)$ whose closure does not contain $(\lambda_0, \alpha_0)$. By definition of singular ambiguities, we may arrange for the boundary of $V_{\lambda, \alpha}$ to contain no singular ambiguities. Define $\OO = \OO ' \backslash \overline{V}$ where $\overline{V}$ is the closure of the union of the $V_{\lambda, \alpha}$. Now $\OO $ satisfies property (c) and inherits the remaining properties from $\OO '$.
\end{proof}

\begin{Def}
If $\OO $ is the set obtained from the previous lemma, we say that the \emph{center} of $\OO $ is $(\lambda_0, \alpha_0)$.
\end{Def}

\begin{Lemma}\label{FiniteCoverLem}
Let $\varepsilon > 0$. Then there is an open cover $\OO _1, ..., \OO _r$ of the singular ambiguities of $E_\CC$ such that each $\OO _k$ is connected and satisfies the following:
\begin{enumerate}
    \item The diameter of $\OO _k$ is less than $\varepsilon$.
    
    \item There is a singular ambiguity $(\lambda_k, \alpha_k) \in \OO _k$ such that any $\CC$-eigenpath $\gamma$ passing through $\overline{\OO _k}$ satisfies $\gamma(\alpha_k) = \lambda_k$.
    
    \item The boundary $\partial \OO _k$ of $\OO _k$ contains finitely many points, none of which are singular ambiguities.
\end{enumerate}
\end{Lemma}

\begin{proof}
Let $(\lambda_0, \alpha_0) \in E_\CC$ be a singular ambiguity, so by Lemma \ref{SeparationLem} we may find a connected open neighborhood $\OO _{\lambda_0, \alpha_0}$ of $(\lambda_0, \alpha_0)$ satisfying properties (a) through (c). Repeating this process for each singular ambiguity, the resulting open sets cover the set of all singular ambiguities. Then by Lemma \ref{CompactSALem}, there is a finite subcover $\OO _1,...,\OO _r$.
\end{proof}

\begin{Lemma}\label{EigenvalueDeltaLem}
Let $\varepsilon>0$ and let $\CC$ be a matrix path. Then there is a $\delta > 0$ such that for any matrix path $\CC'$ with $\|\CC_\alpha - \CC_\alpha'\|<\delta$ for all $\alpha \in [0,1]$, there are orderings $\lambda_{\alpha,1},...,\lambda_{\alpha,n}$ and $\lambda_{\alpha,1}',...,\lambda_{\alpha,n}'$ of the eigenvalues of each $\CC_\alpha$ and $\CC_\alpha'$ such that $|\lambda_{\alpha,j}-\lambda_{\alpha,j}'|$ for $\alpha \in [0,1]$ and $j=1,...,n$.
\end{Lemma}

\begin{proof}
Since the uniform norm $\|\cdot\|$ is equivalent to other generalized matrix norms, it suffices to prove this lemma for the Frobenius norm $\|\cdot\|_2$ given by $\|(a_{i,j})\|_2 = \sum_{i,j} |a_{i,j}|^{1/2}$.

Let $m = \max\{\|\CC_\alpha\|_2 \mid \alpha \in [0,1]\}$ and $\delta = \min \left\{ 2m, \frac{\varepsilon^n}{2^{4n-3}m^{n-1}} \right\}$. Let $\CC'$ be a matrix path $[0,1] \to M_n$ such that $\|\CC_\alpha - \CC_\alpha'\|_2 < \delta$ for all $\alpha \in [0,1]$. For $j=1,...,n$ denote by $\lambda_{\alpha, j}$ and $\lambda_{\alpha, j}'$ the eigenvalues of $\CC_\alpha$ and $\CC_\alpha'$, respectively. Then by the bound in \cite{Bounds}, we may reorder these eigenvalues so that
\begin{align*}
    |\lambda_{\alpha, j} - \lambda_{\alpha, j}'| \leq 4 \times 2^{-1/n} \left(\|C_{\alpha}\|_2 + \|C_{\alpha}'\|_2 \} \right)^{1-1/n} \| C_{\alpha} - C_{\alpha}' \|_2^{1/n},
\end{align*}
and so
\begin{align*}
    |\lambda_{\alpha, j} - \lambda_{\alpha, j}'|^n &< \frac{4^n}{2} \left( 2m + \delta \right)^{n-1} \delta \\
    &\leq 2^{2n-1} (4m)^{n-1} \delta \\
    &\leq 2^{4n-3} m^{n-1} \times \frac{\varepsilon^n}{2^{4n-3} m^{n-1}} \\
    &= \varepsilon^n,
\end{align*}
giving us the bound
\begin{align*}
    |\lambda_{\alpha, j} - \lambda_{\alpha, j}'| < \varepsilon
\end{align*}
for all $\alpha \in [0,1]$ and $j=1,...,n$.
\end{proof}

We now have the machinery necessary to prove Theorem \ref{ClosePathsThm}.

\begin{proof}
Let $\OO _1, ..., \OO _r$ be the open cover of Lemma \ref{FiniteCoverLem} such that each $\OO _k$ has diameter less than $\frac{\varepsilon}{2}$. Define $F = E_\CC \backslash \left( \bigcup_{k=1}^r \OO _k \right)$, so $F$ consists of closed connected components such that $\mult(\lambda, \alpha)$ is constant for all $(\lambda, \alpha)$ in the component. Since the $\OO _k$ have finite boundaries that contain no singular ambiguities, we know that the boundary of $F$ shares this property. Hence $F$ consists of finitely many such connected components $F_1,...,F_s$. For each pair $i,j$ with $1 \leq i < j \leq s$, define $\varepsilon_{i,j} = \inf\{|\lambda_i-\lambda_j| \mid (\lambda_i, \alpha) \in F_i, \; (\lambda_j, \alpha) \in F_j\}$. Then $\varepsilon_{i,j} > 0$, so we set $\varepsilon_0 = \min_{i,j} \varepsilon_{i,j}$. Since each $\partial \OO _k$ contains no singular ambiguities, there is some $\delta_k > 0$ such that $|\lambda - \lambda'| > \delta_k$ for all $\lambda, \lambda'$ such that $(\lambda, \alpha) \in \OO _k$ and $(\lambda', \alpha) \in E_\CC \backslash \OO _k$. Let $\varepsilon_0' = \min\{\delta_1, ..., \delta_r\}$, and define $\varepsilon_* = \frac{1}{2} \min\{\varepsilon_0, \varepsilon_0', \varepsilon\}$.

Now by Lemma \ref{EigenvalueDeltaLem} there is a $\delta > 0$ such that any matrix path $\CC_\alpha'$ that is $\delta$-close to $\CC_\alpha$ has eigenvalues that are $\varepsilon_*$-close to those of $\CC_\alpha$. Let $\CC'$ be such a path, and for $j=1,...,n$ denote by $\lambda_{\alpha, j}$ and $\lambda_{\alpha, j}'$ the eigenvalues of $\CC_\alpha$ and $\CC_\alpha'$, respectively. Then we have $|\lambda_{\alpha, j} - \lambda_{\alpha, j}'| < \varepsilon_*$ for all $j=1,...,n$ and $\alpha \in [0,1]$.

Take $\gamma_j'(\alpha) = \lambda_{\alpha, j}'$ to be continuous, so $\{\gamma_1', ..., \gamma_n'\}$ is a $\CC'$-eigenpath set. We will now find a ``matching" $\CC$-eigenpath set $\{\gamma_1, ..., \gamma_n\}$ by defining each $\gamma_j$ piecewise. Define each $\gamma_j(0)$ such that $|\gamma_j(0) - \gamma_j'(0)| < \varepsilon_*$ for $j=1,...,n$. Since $\varepsilon_* \leq \frac{1}{2}\varepsilon_0, \frac{1}{2}\varepsilon_0'$, we may continue to define each path $\gamma_j$ in such a way that $|\gamma_j(\alpha) - \gamma_j'(\alpha)| < \varepsilon_*$ until some $\gamma_j$ passes through some $\OO _k$. Then, since the diameter of each $\OO _k$ is less than $\frac{\varepsilon}{2}$ and since $\varepsilon_* \leq \frac{\varepsilon}{2}$, we have $|\gamma_i(\alpha) - \gamma_j'(\alpha)| < \varepsilon$ for any $\gamma_i, \gamma_j$ passing through $\OO _k$. It follows that we may define every path $\gamma_j$ to satisfy $|\gamma_j(\alpha) - \gamma_j'(\alpha)| < \varepsilon$ for all $\alpha \leq \alpha_k$ where $(\lambda_k, \alpha_k)$ is the first center that $\gamma_j$ must hit. Since all paths that pass through a certain $\OO _k$ must hit its center by property (b) of Lemma \ref{FiniteCoverLem}, this $\alpha_k$ is well-defined for each $\gamma_j$.

Now order $\alpha_1, ..., \alpha_r$ such that $\alpha_1 \leq ... \leq \alpha_r$. We will proceed by induction. Suppose that each $\gamma_j$ is defined up to one of the points $\alpha_1, ..., \alpha_r$, and all are defined up to at least $\alpha_{k-1}$. Further, suppose that if some $\gamma_j$ is defined up to some $\alpha_\ell$, then $\gamma_j$ hits the center $(\lambda_\ell, \alpha_\ell)$, and no other path that may hit this center is defined past $\alpha_\ell$. Let $\gamma_1, ..., \gamma_t$ be the paths that hit the center $(\lambda_{k-1}, \alpha_{k-1})$. Since $\varepsilon_* \leq \frac{\varepsilon_0}{2}, \frac{\varepsilon_0'}{2}$, we may use similar reasoning as in the above paragraph to define these paths from $\alpha_{k-1}$ to the next center they hit (or to 1) in such a way that $|\gamma_j(\alpha) - \gamma_j'(\alpha)| < \varepsilon$ for all $\alpha$ up to the next center hit by $\gamma_j$. Then all conditions in the inductive hypothesis still hold, but all paths are now defined up to at least $\alpha_k$. It follows that we may continue to define these paths up to $\alpha=1$, finishing our proof.
\end{proof}

\begin{Cor}\label{CommonPairingCor}
There is a $\delta > 0$ such that for all continuous paths $\CC'$ from $A$ to $B$ satisfying $\|\CC_\alpha - \CC_\alpha'\|  < \delta$ for all $\alpha \in [0,1]$, any $\CC'$-eigenpairing $p : \spec(A) \to \spec(B)$ is also a $\CC$-eigenpairing.
\end{Cor}

\begin{proof}
If either $A$ or $B$ has just one distinct eigenvalue, then the claim is trivial. Otherwise let $\varepsilon_A$ be the minimum distance between distinct eigenvalues of $A$, and similarly define $\varepsilon_B$. Now define $\varepsilon = \min\{\varepsilon_A, \varepsilon_B\}$, and obtain the $\delta$ of Theorem \ref{ClosePathsThm}. Then for any continuous path $\CC'$ from $A$ to $B$ satisfying $\|\CC_\alpha - \CC_\alpha'\|  < \delta$ for all $\alpha \in [0,1]$, we may find for any $\CC'$-eigenpath set $\{\gamma_1', ..., \gamma_n'\}$ a $\CC$-eigenpath set $\{\gamma_1, ..., \gamma_n\}$ such that $|\gamma_j(\alpha) - \gamma_j'(\alpha)| < \varepsilon$ for all $\alpha \in [0,1]$ and $j=1,...,n$. It follows that $\gamma_j(0) = \gamma_j'(0)$ and $\gamma_j(1) = \gamma_j'(1)$ for each $j$, so these path sets induce the same pairing.
\end{proof}


\section{Proof of Theorem 2}

In the previous section, we saw that small perturbations of the matrix path $\CC$ yield small perturbations of the corresponding eigenpaths. Dual to this notion is whether there exists an arbitrarily small perturbation of $\CC$ that admits no ambiguities and whose unique eigenpath set approximates a given $\CC$-eigenpath set. This section is dedicated to proving the validity of this statement, which indicates that there is no canonical choice of $\CC$-eigenpath set without imposing supplementary conditions.

To complete this proof, we will first approximate $\CC$ by a polynomial matrix path $\PP$ that coincides with $\CC$ at certain important points (namely, the center singular ambiguities of the cover from Lemma \ref{FiniteCoverLem}). This approximation will be such that any $\CC$-eigenpath set can be approximated by a $\PP$-eigenpath set. Then it will suffice to prove the theorem in the case that $\CC$ is a polynomial matrix path with no nonsingular ambiguities. For this case, we will provide a construction that allows us to ``rip apart" the singular ambiguities into non-intersecting eigenpaths.

\begin{Lemma}\label{WeierstrassLem}
Let $\varepsilon > 0$ and let $\CC$ be a matrix path. Further, let $\alpha_1, ..., \alpha_r \in [0,1]$. Then there is a matrix path $\PP$ such that
\begin{enumerate}
    \item the entries of $\PP_\alpha$ are complex polynomials in $\alpha$.
    
    \item for $k=1,...,r$, we have $\PP_{\alpha_k} = \CC_{\alpha_k}$.
    
    \item for $\alpha \in [0,1]$, we have $\|\PP_\alpha - \CC_\alpha\| < \varepsilon$.
    
    \item all ambiguities in $E_\PP$ are singular.
\end{enumerate}
\end{Lemma}

\begin{proof}
Let $\alpha_0 \in [0,1] \backslash\{\alpha_1, ..., \alpha_r\}$. Since the set $M_n'$ of matrices with no repeated eigenvalues in dense in $M_n$, we may find $A \in M_n'$ such that $\|A-\CC_{\alpha_0}\| < \frac{\varepsilon}{2}$. Now define the matrix path $\CC'$ to coincide with $\CC$ except on a small neighborhood of $\alpha_0$, where it deviates by less than $\frac{\varepsilon}{2}$ to hit $A$. We then use the Weierstrass approximation theorem to give a uniform $\frac{\varepsilon}{2}$-approximation of each entry of $\CC_\alpha'$ while fixing the values at $\alpha_0, \alpha_1, ..., \alpha_r$. These approximations form a polynomial matrix path $\PP$ satisfying (a), (b), and (c). Since $\PP_{\alpha_0}$ has $n$ distinct eigenvalues, property (d) also follows by Lemma \ref{PathMultiplicityLem}.
\end{proof}

\begin{Lemma}\label{ConvexApproxLem}
Let $\varepsilon > 0$, and let $A, B \in M_n$ with $C_\alpha = (1-\alpha)A + \alpha B$ the convex path between them. If $A$ and $B$ have distinct eigenvalues, then there is a matrix path $\CC$ from $A$ to $B$ such that $E_\CC$ contains no ambiguities and $\|\CC_\alpha - C_\alpha\| < \varepsilon$ for all $\alpha \in [0,1]$.
\end{Lemma}

\begin{proof}
Let $d$ be the discriminant of the characteristic polynomial of $C_\alpha$, so $d$ is a polynomial in $\alpha$. Then $C_\alpha$ has a repeated eigenvalue if and only if $d(\alpha) = 0$, which occurs at finitely points $\alpha \in \C$. Define $\CC_\alpha$ by traversing the convex path $C_\alpha$, except taking $\varepsilon$-small detours around any roots $\alpha \in [0,1]$ of $d$. Then each $\CC_\alpha$ has distinct eigenvalues, so we are finished.
\end{proof}

\begin{Lemma}\label{DiagonalPairingLem}
Let $D \in M_n$ be diagonal with distinct diagonal entries $\lambda_1, ..., \lambda_n$, including some $\lambda_i \neq \lambda_j$. Then there is a matrix path $\CC$ from $D$ to itself such that $E_\CC$ has no ambiguities, the unambiguous $\CC$-eigenpairing swaps $\lambda_i$ and $\lambda_j$ but fixes the rest of the eigenvalues, and $\|\CC_\alpha - D\| < |\lambda_i - \lambda_j|$ for all $\alpha \in [0,1]$.
\end{Lemma}

\begin{proof}
By employing a change of basis, we may assume without loss of generality that $i=1$ and $j=2$. Now for $\alpha \in [0, \frac{1}{2}]$ let $\SSS_\alpha$ be the direct sum of the rotation matrix
$$\begin{bmatrix} \cos \pi \alpha & -\sin \pi \alpha \\ \sin \pi \alpha & \cos \pi \alpha \end{bmatrix}$$
with the identity matrix of size $(n-2)$-by-$(n-2)$. Then for $\alpha \in [0, \frac{1}{2}]$ let $\CC_\alpha = \SSS_\alpha D \SSS_\alpha^{-1}$, so $\CC_0 = D$ and $\CC_\frac{1}{2}$ is $D$ but with $\lambda_1$ and $\lambda_2$ swapped. Since the spectrum of $\CC_\alpha$ remains constant on this path, it follows that there are no ambiguities for $\alpha \in [0, \frac{1}{2}]$. Further, the unambiguous $\CC$-eigenpairing of $\CC_0$ and $\CC_\frac{1}{2}$ is the identity map $\lambda_k \mapsto \lambda_k$.

We note that $\CC_\alpha$ is the direct sum of
$$\frac{1}{2}\begin{bmatrix} \lambda_1 + \lambda_2 + (\lambda_1-\lambda_2)\cos 2\pi \alpha & (\lambda_1 - \lambda_2)\sin 2\pi \alpha \\ (\lambda_1 - \lambda_2)\sin 2\pi \alpha & \lambda_1 + \lambda_2 + (\lambda_2-\lambda_1)\cos 2\pi \alpha \end{bmatrix}$$
with the blocks $[\lambda_3], ..., [\lambda_n]$. Thus, we see that no entry of $\CC_\alpha$ differs by more than $|\lambda_1-\lambda_2|$ from the corresponding entry in $D$ on $[0, \frac{1}{2}]$.

Now since there are finitely many other eigenvalues $\lambda_3, ..., \lambda_n$, we may find paths $[\frac{1}{2}, 1] \to \C$ from $\lambda_1$ to $\lambda_2$ and $\lambda_2$ to $\lambda_1$ that disagree everywhere on $[\frac{1}{2}, 1]$ and never hit any $\lambda_3, ..., \lambda_n$. Further, we may assume that the images of each path is contained within the disk of radius $\frac{|\lambda_1-\lambda_2|}{2}$ centered on $\frac{\lambda_1+\lambda_2}{2}$. Define $\CC_\alpha$ on $[\frac{1}{2}, 1]$ as a diagonal matrix whose first two entries are precisely these paths, so then $\CC_1 = D$. Then the desired swap of $\lambda_1$ and $\lambda_2$ is the unambiguous $\CC$-eigenpairing of $D$ and itself, and we have $\|\CC_\alpha - D\| < |\lambda_1 - \lambda_2|$ for all $\alpha \in [0,1]$.
\end{proof}

\begin{Lemma}\label{PolynomialRippingLem}
Let $\varepsilon > 0$ and let $\PP$ be a polynomial matrix path such that all ambiguities in $E_\PP$ are singular. Further, let $\{\gamma_1, ..., \gamma_n\}$ be a $\PP$-eigenpath set. Then there is a matrix path $\CC$ admitting a unique $\CC$-eigenpath set $\{\gamma_1', ..., \gamma_n'\}$ such that $\|\PP_\alpha - \CC_\alpha\| < \varepsilon$ and $|\gamma_j(\alpha) - \gamma_j'(\alpha)| < \varepsilon$ for all $\alpha \in [0,1]$ and $j=1,...,n$.
\end{Lemma}

\begin{proof}
Since the supremum norm $\|\cdot\|$ is equivalent to the basis-invariant Frobenius norm $\|\cdot\|_2$, it suffices to prove this lemma for the Frobenius norm.

Enumerate the points $\alpha_1 < ... < \alpha_r$ at which singular ambiguities occur in $E_\PP$, and find a collection of disjoint open intervals $U_1', ..., U_r' \subseteq [0,1]$ such that $\alpha_k \in U_k'$ for $k=1,...,r$. For each $k$, there is an open subinterval $U_k \subseteq U_k'$ that still contains $\alpha_k$ and such that $\|\PP_\alpha - \PP_{\alpha_k}\|_2 < \frac{\varepsilon}{2}$ and $|\gamma_j(\alpha) - \gamma_j(\alpha_k)| < \frac{\varepsilon}{2}$ for all $\alpha \in U_k$ and $j=1,...,r$. Since we may re-parameterize $\PP_\alpha$ by extending the domain $[0,1]$ to a slightly larger real interval without significantly changing the entries at any given point, we may assume without loss of generality that $\alpha_1 > 0$ and $\alpha_r < 1$. With this assumption, we may further assume that each $U_k = (\alpha_k^-, \alpha_k^+)$ for some $\alpha_k^-, \alpha_k^+$ satisfying $0 < \alpha_k^- < \alpha_k < \alpha_k^+ < 1$.

We now define $\CC$ as follows: Set $\CC_\alpha = \PP_\alpha$ outside $\bigcup_k U_k$. For $\alpha \in U_k$, beginning with $k=1$, we use the following procedure to define $\CC_\alpha$. Since $\PP_{\alpha_1^-}$ has $n$ distinct eigenvalues, we may choose a basis of $\C^n$ such that $\PP_{\alpha_1^-} = D$ is diagonal. That is, we have
\begin{align*}
    D = \begin{bmatrix}
    \gamma_1(\alpha_1^-) & & \\
     & \ddots & \\
     & & \gamma_n(\alpha_1^-)
     \end{bmatrix}.
\end{align*}
Now set $\CC_{\alpha_1} = D$, and use Lemma \ref{ConvexApproxLem} to define $\CC_\alpha$ on $(\alpha_1, \alpha_1^+)$ as $\frac{\varepsilon}{2}$-close to the convex path $C$ from $D$ to $\PP_{\alpha_1^+}$, but without any ambiguities. Then there is some unique $\CC$-eigenpairing $p$ from $\CC_{\alpha_1}$ to $\CC_{\alpha_1^+}$, so there must be a permutation $\tau$ of the eigenvalues of $D$ such that $p \circ \tau$ is the desired $\CC$-eigenpairing $\gamma_j(\alpha_1^-) \mapsto \gamma_j(\alpha_1^+)$ of $\CC_{\alpha_1^-}$ and $\CC_{\alpha_1^+}$. Then $\tau$ is a product of transpositions involving $\frac{\varepsilon}{2}$-close eigenvalues, so we may apply Lemma \ref{DiagonalPairingLem} once for each transposition to construct a path from $D$ to itself. In particular, this path achieves the desired unambiguous eigenpairing $\tau$ and never deviates by more than $\frac{\varepsilon}{2}$ from $D$.

For $\alpha \in (\alpha_1^-, \alpha_1)$, we have
\begin{align*}
    \|\CC_\alpha - \PP_\alpha\|_2 &= \|\CC_\alpha - D + D- \PP_\alpha\|_2 \\
    &\leq \|\CC_\alpha - D\|_2 + \|D- \PP_\alpha\|_2 \\
    &< \frac{\varepsilon}{2} + \frac{\varepsilon}{2} \\
    &= \varepsilon.
\end{align*}
and for $\alpha \in (\alpha_1, \alpha_1^+)$ we have
\begin{align*}
    \|\CC_\alpha - \PP_\alpha\|_2 &= \|\CC_\alpha - C_\alpha + C_\alpha- \PP_\alpha\|_2 \\
    &\leq \|\CC_\alpha - C_\alpha\|_2 + \|C_\alpha- \PP_\alpha\|_2 \\
    &< \frac{\varepsilon}{2} + \frac{\varepsilon}{2} \\
    &= \varepsilon.
\end{align*}
We may then repeat this process for $k=2,...,r$, finishing our proof.
\end{proof}

We now give our proof of Theorem \ref{RippingThm}.

\begin{proof}
Let $\OO _1, ..., \OO _r$ be the open cover of the singular ambiguities determined by Lemma \ref{FiniteCoverLem}, each with diameter less than $\varepsilon$ and center $(\lambda_k, \alpha_k)$. We then obtain $\varepsilon_*  = \frac{1}{2}\min\{\varepsilon_0, \varepsilon_0', \frac{\varepsilon}{2}\}$ similarly as in Theorem \ref{ClosePathsThm}, so by Lemma \ref{EigenvalueDeltaLem} there is a $\delta>0$ such that a $\delta$-perturbation of the matrix path $\CC$ induces at most an $\varepsilon_*$-perturbation of the eigenvalues at any point. Assume without loss of generality that $\delta < \frac{\varepsilon}{2}$. Then by Lemma \ref{WeierstrassLem}, we may find a polynomial matrix path $\PP$ that agrees with $\CC$ on $\alpha_1, ..., \alpha_r$, admits only singular ambiguities, and $\delta$-approximates $\CC$.

Now consider the matrix path homotopy given by $H_{\alpha, \beta} = (1-\beta) \CC_\alpha + \beta \PP_\alpha$. As we shift $\beta$ from 0 to 1, the $\CC$-eigenregion continuously deforms into the $\PP$-eigenregion, each point shifting by no more than $\varepsilon_*$.  By our construction of $\varepsilon_*$, the shifts of the $\OO _1, ..., \OO _r$ are well-defined in the following sense: for any fixed $\OO _k$ and $\alpha_0 \in [0,1]$, let $\Lambda_{k, \alpha_0} = \{\lambda \in \C \mid (\lambda, \alpha_0) \in \OO _k\}$ be a multiset wherein each point $\lambda$ occurs $\mult(\lambda, \alpha_0)$ times. Then for any continuous parameterizations of the points $\lambda \in \Lambda$ under the shift by the homotopy, the total shifted multiset $\Lambda$ is the same as it would be for any other parameterization. Similarly, the shifts of the connected components $F_1, ..., F_s$ of $E_\CC \backslash \{\bigcup_{k=1}^r \OO _k\}$ are also well-defined in the same sense.

Note that each $\gamma_j$ is determined (up to variation by the maximum diameter $\frac{\varepsilon}{4}$ of the $\OO _k$) entirely by which of the $\OO _k$ (for $k=1,...,r$) and $F_\ell$ (for $\ell=1,...,s$) it passes through. Since the shifts of these objects under the homotopy $H_{\alpha, \beta}$ are well-defined, we may find a $\PP$-eigenpath set $\{\eta_1, ..., \eta_n\}$ such that each $\eta_j$ passes through the same (shifted versions of) $\OO _1, ..., \OO _r, F_1, ..., F_s$ as does $\gamma_j$. Then $|\eta_j(\alpha) - \gamma_j(\alpha)| < \varepsilon_* \leq \frac{\varepsilon}{4}$ when $\alpha$ is such that these paths are passing through a component of the form $F_\ell$. Since the diameter of each $\OO _k$ is at most $\frac{\varepsilon}{4}$, we have $|\eta_j(\alpha) - \gamma_j(\alpha)| < \varepsilon_* + \frac{\varepsilon}{4} \leq \frac{\varepsilon}{2}$ when $\alpha$ is such that these paths are passing through a component of the form $\OO _k$. Thus $|\eta_j(\alpha) - \gamma_j(\alpha)| < \frac{\varepsilon}{2}$ and $\|\CC_\alpha - \PP_\alpha\| < \frac{\varepsilon}{2}$ for all $\alpha \in [0,1]$.

We may now apply Lemma \ref{PolynomialRippingLem} to find an $\frac{\varepsilon}{2}$-approximation $\CC'$ to $\PP$ that admits no ambiguities and whose unique $\CC'$-eigenpath set $\{\gamma_1',...,\gamma_n'\}$ satisfies $|\eta_j(\alpha) - \gamma_j'(\alpha)| < \frac{\varepsilon}{2}$. Then $\CC'$ is an $\varepsilon$-approximation of $\CC$, and each $\gamma_j'$ is an $\varepsilon$-approximation of $\gamma_j$, so we are finished.
\end{proof}

\begin{Cor}
Along with the hypotheses of Theorem \ref{RippingThm}, suppose that $\CC_0$ and $\CC_1$ have $n$ distinct eigenvalues. Then the ambiguity-free approximation $\CC'$ may be chosen so that $\CC'_0 = \CC_0$ and $\CC'_1 = \CC_1$. If only one (say $\CC_0$) has $n$ distinct eigenvalues, then the approximation may still be chosen so that $\CC_0' = \CC_0$.
\end{Cor}

\begin{proof}
Use Lemma \ref{WeierstrassLem} to approximate $\CC$ with a polynomial path $\PP$ that also agrees with $\CC$ at $\alpha=0,1$. Then by our construction of $\CC'$ in Lemma \ref{PolynomialRippingLem}, we see that $\CC'$ agrees with $\PP$ (and so also with $\CC$) at $\alpha=0,1$. Note that this procedure still works if we want only $\CC'_0 = \CC_0$ or $\CC'_1 = \CC_1$.
\end{proof}


\section{Proof of Theorem 3}

In this section we shift our focus from eigenpaths to eigenpairings. As will be seen in section \ref{Section2x2}, low-dimensional convex eigenpairings are relatively easy to determine. Here we prove Theorem 3, which allows us to predict eigenpairings for a slightly more general class of matrix paths.

We begin by proving the simple case wherein $f$ and $g$ are everywhere nonnegative, and then use a technical lemma along with Theorem \ref{ClosePathsThm} to reduce the problem to the aforementioned simple case. We will again denote the convex eigenpath by $C_\alpha = (1-\alpha)A + \alpha B$.

\begin{Lemma}\label{ConeLem}
Let $\CC_\alpha = f(\alpha)A+g(\alpha)B$ be a path from $A$ to $B$ for continuous $f,g : [0,1] \to \R$. If $f(\alpha), g(\alpha) \geq 0$ for all $\alpha \in [0,1]$, then any  convex eigenpairing $p$ is also a $\CC$-eigenpairing.
\end{Lemma}

\begin{proof}
We know that each $\CC_\alpha$ is equal to $c_\alpha C_{\beta(\alpha)}$ for some $c_\alpha \geq 0$ and convex combination $C_{\beta(\alpha)}$. Further, we may view $c_\alpha$ and $\beta(\alpha)$ as continuous functions of $\alpha$. We then obtain
\begin{align*}
    E_\CC &= \{(\lambda, \alpha) \mid \alpha \in [0,1], \; \lambda \in \sigma(\CC_\alpha)\} \\
    &= \{(c_\alpha \lambda, \alpha) \mid \alpha \in [0,1], \; \lambda \in \sigma(C_{\beta(\alpha)})\}.
\end{align*}
If $\{\gamma_1, ..., \gamma_n\}$ is a convex eigenpath set that induces the convex eigenpairing $p$, then the continuous functions $\gamma_j'$ given by $\gamma_j'(\alpha) = c_\alpha \gamma_j(\beta(\alpha))$ form a $\CC$-eigenpath set. Note that $\gamma_j'(0) = \gamma_j(0)$ and $\gamma_j'(1) = \gamma_j(1)$, so $p$ is a $\CC$-eigenpairing as well.
\end{proof}

\begin{Lemma}\label{ArcLem}
Let $\CC_\alpha = f(\alpha)A+g(\alpha)B$ be a path from $g(0)B$ to $g(1)B$ for continuous $f,g : [0,1] \to \R$. If $f(\alpha), g(\alpha) \geq 0$ for all $\alpha \in [0,1]$, then the bijection $\spec(g(0)B) \to \spec(g(1)B)$ given by $\lambda \mapsto \frac{g(1)}{g(0)} \lambda$ is a $\CC$-eigenpairing.
\end{Lemma}

\begin{proof}
If $f$ and $g$ simultaneously vanish at some point in $[0,1]$, then we are finished by Lemma \ref{AllBijectionsCor}. Otherwise choose $\alpha_0 \in [0,1]$ such that $g(\alpha_0) = 0$, or if no such zero exists choose $\alpha_0$ that maximizes $\frac{f(\alpha_0)}{g(\alpha_0)}$.

Now let $p$ be a convex eigenpairing of $B$ and $C_{\alpha_0}$. Notice that for every $\alpha \in [0, 1]$, we know that $\CC_\alpha$ is a positive $\R$-linear combination of $B$ and $C_{\alpha_0}$. Then by Lemmas \ref{PosScalLem} and \ref{ConeLem}, we find that a suitably scaled version of $p$ is a $\CC$-eigenpairing of $g(0)B$ and $C_{\alpha_0}$, and similarly that a scaled version of $p^{-1}$ is a $\CC$-eigenpairing of $C_{\alpha_0}$ and $g(1)B$. Then by composing these maps via Lemma \ref{Concatenation}, we obtain the desired result.
\end{proof}

We may now complete the proof of Theorem \ref{ConvexReductionThm}.

\begin{proof}
Let $\{\gamma_1, ..., \gamma_n\}$ be a convex eigenpath set corresponding to a convex eigenpairing $p$. First observe that if there is some $\alpha_0 \in [0,1]$ such that $f(\alpha_0) = 0 = g(\alpha_0)$, then $\CC_{\alpha_0} = 0$ and so every bijection is a $\CC$-eigenpairing by Lemma \ref{AllBijectionsCor}. Henceforth we will assume there is no such $\alpha_0$.

By Corollary \ref{CommonPairingCor}, there is a $\delta>0$ such that for all continuous paths $\CC_\alpha'$ from $A$ to $B$ satisfying $\|\CC_\alpha - \CC_\alpha'\|  < \delta$ for all $\alpha \in [0,1]$, any $\CC'$-eigenpairing $p : \spec(A) \to \spec(B)$ is also a $\CC$-eigenpairing. Let $U$ be the union of all intervals $(\alpha_1, \alpha_2) \subset [0,1]$ on which $f$ is negative and such that $f(\alpha_1) = 0 = f(\alpha_2)$ and $f(\alpha) > -\frac{\delta}{\|A\|}$ for all $\alpha \in (\alpha_1, \alpha_2)$. Similarly, let $V$ be the union of all intervals $(\alpha_1, \alpha_2) \subset [0,1]$ on which $g$ is negative and such that $g(\alpha_1) = 0 = g(\alpha_2)$ and $g(\alpha) > -\frac{\delta}{\|B\|}$ for all $\alpha \in (\alpha_1, \alpha_2)$. If we define $f'=f$ on $[0,1]\backslash U$ and $f'=0$ on $U$ and similarly $g'=g$ on $[0,1]\backslash V$ and $g'=0$ on $V$, then the resulting path $\CC'_\alpha = f'(\alpha)A + g'(\alpha)B$ satisfies $\|\CC_\alpha - \CC_\alpha'\|  < \delta$ for all $\alpha \in [0,1]$. It therefore suffices to show that any convex eigenpairing $p$ is also a $\CC'$-eigenpairing.

Let $(\alpha_1, \alpha_2) \subset [0,1]$ be an interval on which $f'$ is negative and such that $f'(\alpha_1) = 0 = f'(\alpha_2)$. Then $g'(\alpha) \geq 0$ for all $\alpha \in (\alpha_1, \alpha_2)$, so by Lemma \ref{ArcLem} we know that a scaled version of the identity map is a $\CC'$-eigenpairing between $g'(\alpha_1)B$ and $g'(\alpha_2)B$. This bijection would be the only $\CC'$-eigenpairing if we redefined $f'$ to be identically zero and $g'$ to be the straight line between $g'(\alpha_1)B$ and $g'(\alpha_2)B$ on $(\alpha_1, \alpha_2)$, so we may assume without loss of generality that $f'$ is nonzero on this interval.

By our construction of $\CC'$, we know there are finitely many such intervals on which $f'$ is negative. We may therefore repeat this process for all these intervals, so in fact we may assume that $f'(\alpha) \geq 0$ for all $\alpha \in [0,1]$. By an identical argument, we may assume that $g'(\alpha) \geq 0$ as well. Thus, the theorem follows by Lemma \ref{ConeLem}.
\end{proof}


\section{The $2$-by-$2$ Case of Convex Eigenpairings}\label{Section2x2}
Before we investigate the $2$-by-$2$ case, we will first prove a few results that apply to the general convex case. First, we determine under what conditions there exist \emph{straight line} convex eigenpaths, i.e.\ paths that are degree-one polynomials in $\alpha$.

\begin{Lemma}\label{SameEvecsLem}
If $\lambda \in \sigma(A)$ and $\mu \in \sigma(B)$ share an eigenvector $v \in \C^n$, then the straight line given by $\gamma(\alpha) = (1-\alpha)\lambda + \alpha \mu$ is a convex eigenpath.
\end{Lemma}

\begin{proof}
We have $C_\alpha v = ((1-\alpha)\lambda + \alpha \mu)v$, so $\gamma$ is indeed an eigenpath.
\end{proof}

\begin{Lemma}\label{TraceCondition}
Let $\{\gamma_1, ..., \gamma_n\}$ be an eigenpath set. Then the pointwise sum $\gamma = \sum_{j=1}^n \gamma_j$ is a straight line.
\end{Lemma}

\begin{proof}
Notice that $\gamma(\alpha) = \Tr(C_\alpha) = (1-\alpha)\Tr(A) + \alpha \Tr(B)$, so $\gamma$ is indeed a straight line.
\end{proof}

\begin{Rmk}
In view of Lemma \ref{Combination}, we see that if $A$ and $B$ share some $k$ linearly independent eigenvectors, then we may reduce the convex eigenpairing problem to the $(n-k)$-by-$(n-k)$ case by simply using a basis in which $A$ and $B$ are block upper triangular. In particular, there will be $k$ blocks of size $1$-by-$1$, each of which corresponds to a shared eigenvector. The remaining $(n-k)$-by-$(n-k)$ block may then be treated separately.
\end{Rmk}

We now proceed to our analysis of the $2$-by-$2$ case: Suppose that $A$ and $B$ are $2$-by-$2$ complex matrices with $\spec(A) = \{\lambda_1, \lambda_2\}$ and $\spec(B) = \{\mu_1, \mu_2\}$. Notice that the possible eigenpairings are $\lambda_j \mapsto \mu_j$ (denoted by $p$) and $\lambda_j \mapsto \mu_{3-j}$ (denoted by $q$). In the remainder of this section, we will analyze the conditions under which each of these eigenpairings may occur.

If either has an eigenvalue of algebraic multiplicity 2, then both $p$ and $q$ are convex eigenpairings of $A$ and $B$. We may therefore restrict our attention to the case in which $\lambda_1 \neq \lambda_2$ and $\mu_1 \neq \mu_2$. By uniform similarity invariance, it follows that we may reduce to the case in which $A$ is diagonal, so its eigenvectors are $\binom{1}{0}$ and $\binom{0}{1}$. Further, we will assume that the eigenvectors of $B$ are $\binom{v_1}{1}$ and $\binom{v_2}{1}$, corresponding to $\mu_1$ and $\mu_2$, respectively. Later we will address the remaining (trivial) case in which $B$ has an eigenvector of the form $\binom{v}{0}$.

To simplify notation later in this section, we will write $\lambda = \lambda_1-\lambda_2$ and $\mu = \mu_1-\mu_2$. Since pairings are invariant under shifts by the identity matrix (Lemma \ref{PosScalLem}), we note that these quantities have a natural invariance property. We therefore obtain
\begin{align*}
    A &= \begin{pmatrix} 
            \lambda_1 & 0 \\
            0 & \lambda_2
         \end{pmatrix} \\
    B &= \begin{pmatrix} 
            v_1 & v_2 \\
            1 & 1
         \end{pmatrix}
         \begin{pmatrix} 
            \mu_1 & 0 \\
            0 & \mu_2
         \end{pmatrix}
         \begin{pmatrix} 
            v_1 & v_2 \\
            1 & 1
         \end{pmatrix}^{-1} \\
      &= \begin{pmatrix} 
            \frac{\mu_1v_1}{v_1-v_2} - \frac{\mu_2v_2}{v_1-v_2} & -\frac{\mu v_1 v_2}{v_1-v_2} \\
            \frac{\mu}{v_1-v_2} & \frac{\mu_2v_1}{v_1-v_2} -\frac{\mu_1v_2}{v_1-v_2}
         \end{pmatrix} \\
    C_\alpha &= \begin{pmatrix} 
            (1-\alpha)\lambda_1 + \alpha \left( \frac{\mu_1v_1}{v_1-v_2} - \frac{\mu_2v_2}{v_1-v_2} \right) & -\frac{\alpha\mu v_1 v_2}{v_1-v_2} \\
            \frac{\alpha\mu}{v_1-v_2} & (1-\alpha)\lambda_2 + \alpha \left( \frac{\mu_2v_1}{v_1-v_2} -\frac{\mu_1v_2}{v_1-v_2} \right)
         \end{pmatrix}.
\end{align*}
The characteristic polynomial of $C_\alpha$ then has roots
\begin{align*}
    \frac{(1-\alpha)(\lambda_1+\lambda_2) + \alpha(\mu_1+\mu_2) \pm \sqrt{(1-\alpha)^2\lambda^2 + \alpha^2\mu^2 + 2\left(\frac{v_1+v_2}{v_1-v_2} \right)(1-\alpha)\alpha \lambda \mu}}{2},
\end{align*}
so a repeated root occurs precisely when the discriminant $\gamma : [0, 1] \to \C$ given by
\begin{align*}
    \gamma(\alpha) = (1-\alpha)^2\lambda^2 + \alpha^2\mu^2 + 2\left(\frac{v_1+v_2}{v_1-v_2} \right)(1-\alpha)\alpha \lambda \mu
\end{align*}
equals 0 for some $\alpha \in (0,1)$, which is true if and only if
\begin{align}
    \frac{\mu}{\lambda} = \left(\frac{\alpha-1}{\alpha}\right) \left( \frac{v_1+v_2 \pm 2\sqrt{v_1v_2}}{v_1-v_2} \right) \label{JCond}
\end{align}
for that value of $\alpha$.

\begin{Rmk}\label{SqrtPaths}
If the $\gamma(\alpha) \neq 0$ for all $\alpha \in [0,1]$, then by Proposition \ref{AmbiguityProp} there is an unambiguous eigenpairing. Further, there are two distinct continuous ``square root paths" $\eta_1, \eta_2 : [0,1] \to \C$ such that $\eta_1(\alpha)^2 = \gamma(\alpha) = \eta_2(\alpha)^2$ for all $\alpha \in [0,1]$. These paths connect either $\lambda$ to $\mu$ and $-\lambda$ to $-\mu$ or $\lambda$ to $-\mu$ and $-\lambda$ to $\mu$.

If one of these paths (say $\eta_1$) connects $\lambda$ to $\mu$, then one of the eigenpaths that induces the unambiguous eigenpairing is $\gamma_1(\alpha) = \frac{(1-\alpha)(\lambda_1+\lambda_2) + \alpha(\mu_1+\mu_2) + \eta_1(\alpha)}{2}$, so $\gamma_1$ is a path from $\lambda_1$ to $\mu_1$. In this case, it follows that $p$ is the unambiguous eigenpairing. Similarly, if one of the square root paths connects $\lambda$ to $-\mu$, then $q$ is the unambiguous eigenpairing.

If we fix values of $v_1$ and $v_2$ but allow $\frac{\mu}{\lambda}$ to vary, we therefore see that the eigenpaths may swap only when $\gamma$ hits the origin. That is, if $\gamma$ does not hit the origin at some point in a continuous perturbation of $\frac{\mu}{\lambda}$, then the endpoints of the square root paths $\eta_1$ and $\eta_2$ remain the same.
\end{Rmk}

\begin{Lemma}\label{OriginOrNoLem}
Suppose $\frac{\mu}{\lambda} \in \R$, so there is a ray $R$ emanating from the origin that contains $\mu^2$ and $\lambda^2$. Denote by $-R$ the ``opposite ray," the one emanating from the origin that contains $-\mu^2$ and $-\lambda^2$. If the discriminant $\gamma : [0,1] \to \C$ has no roots, then $\gamma$ never hits any point in $-R$.
\end{Lemma}

\begin{proof}
Let $L = R \cup -R$ be the line containing 0, $\mu^2,$ and $\lambda^2$. Notice that $\lambda \mu \in L$, so $\gamma(\alpha)$ is a $\C$-linear combinations of $\lambda^2, \mu^2,$ and $\lambda \mu$. If $\frac{v_1+v_2}{v_1-v_2} \in \R$, then this combination is $\R$-linear, so the image of $\gamma$ lies in $L$. In this case, since $\gamma$ has no roots, we know that $\gamma$ can never cross the origin and hit $-R$.

Otherwise $\frac{v_1+v_2}{v_1-v_2} \notin \R$, in which case the image of $\gamma$ lies in $L$ only for $\alpha=0,1$. Therefore $\gamma$ does not hit $-R$ in this case either.
\end{proof}

\begin{Lemma}\label{RealRatioLem}
Suppose $\frac{\mu}{\lambda} \in \R$.
\begin{enumerate}
    \item If $\frac{\mu}{\lambda} > 0$, then $p$ is an eigenpairing.
    \item If $\frac{\mu}{\lambda} < 0$, then $q$ is an eigenpairing.
\end{enumerate}
\end{Lemma}

\begin{proof}
If $\gamma$ has any roots, then both $p$ and $q$ are eigenpairings and we are finished. Otherwise $\gamma$ has no roots, so we may apply Lemma \ref{OriginOrNoLem} to find that $\gamma$ never hits $-R$. Thus, when finding the square root paths of Remark \ref{SqrtPaths}, we may use the branch cut along the ray $-R$.

First suppose $\frac{\mu}{\lambda} > 0$. Then $\mu$ and $\lambda$ both lie on the same ray emanating from 0, meaning so too do $\mu^2$ and $\lambda^2$. Then each square root path must start and end on the same ray, so one such path traverses from $\lambda$ to $\mu$. By our reasoning in Remark \ref{SqrtPaths}, it follows that $p$ is an eigenpairing. Case (b) is identical.
\end{proof}

In the next several paragraphs, we will assume that arguments of complex numbers lie in $(-\pi, \pi]$. Denote by $\theta$ and $\theta'$ the arguments of $\left(\frac{\alpha-1}{\alpha}\right) \left( \frac{v_1+v_2 \pm 2\sqrt{v_1v_2}}{v_1-v_2} \right)$, where $\theta' \leq \theta$. Notice that this quantity equals 0 (and thus has no argument) only when $v_1=v_2=0$, which would imply that the two distinct eigenvalues of $B$ share an eigenvector.

\begin{Lemma}\label{2x2SymmetryLem}
Either $\theta' = -\theta$ or $\theta = \pi = \theta'$.
\end{Lemma}

\begin{proof}
We have
\begin{align*}
    \left( \frac{v_1+v_2 + 2\sqrt{v_1v_2}}{v_1-v_2} \right)\left( \frac{v_1+v_2 - 2\sqrt{v_1v_2}}{v_1-v_2} \right) &= \frac{(v_1+v_2)^2 - 4v_1v_2}{(v_1-v_2)^2} \\
    &= \frac{(v_1-v_2)^2}{(v_1-v_2)^2} \\
    &= 1,
\end{align*}
so the lemma holds.
\end{proof}

\begin{usethmcounterof}{2x2Theorem}
Define $\mu, \lambda,$ and $\theta$ as above.
\begin{enumerate}
    \item If $|\arg(\frac{\mu}{\lambda})| = \theta$, then both $p$ and $q$ are eigenpairings.
    
    \item If $|\arg(\frac{\mu}{\lambda})| < \theta$, then only $p$ is an eigenpairing.
    
    \item If $|\arg(\frac{\mu}{\lambda})| > \theta$, then only $q$ is an eigenpairing.
\end{enumerate}
\end{usethmcounterof}

\begin{proof}
Notice that condition (a) is equivalent to equation \eqref{JCond} from before, so (a) indeed holds. Now fix some value of $\theta$, and consider the change in the discriminant path $\gamma$ as we continuously shift $\arg(\frac{\mu}{\lambda})$ from $-\pi$ to $\pi$.

If $\theta=0$, then $\gamma$ hits the origin only when $\arg(\frac{\mu}{\lambda}) = 0$, so by our reasoning in Remark \ref{SqrtPaths} it follows that the endpoints of the square root paths may swap at most once through a rotation of $\frac{\mu}{\lambda}$ by $2\pi$. However, a full rotation by $2\pi$ is equivalent to no rotation at all, so in fact the endpoints cannot swap. We know by Lemma \ref{RealRatioLem} that $q$ is an eigenpairing at $\arg(\frac{\mu}{\lambda}) = \pi$ as well, so it is therefore always an eigenpairing in the case that $\theta=0$. It follows by an identical argument that $p$ is always an eigenpairing when $\theta = \pi$.

Otherwise $0 < \theta < \pi$, so the endpoints of the square root paths may swap when $\arg(\frac{\mu}{\lambda}) = \pm \theta$. Again by Lemma \ref{RealRatioLem} we obtain that $q$ is an eigenpairing at $\arg(\frac{\mu}{\lambda}) = \pi$ and $p$ is an eigenpairing at $\arg(\frac{\mu}{\lambda}) = 0$, so in fact the endpoints must swap at $\pm \theta$. Since they cannot swap anywhere else, we are finished.
\end{proof}

\begin{Rmk}
We remark that if $v_1$ has very large magnitude and $v_2$ is very small, then the quantity $\left(\frac{\alpha-1}{\alpha}\right) \left( \frac{v_1+v_2 \pm 2\sqrt{v_1v_2}}{v_1-v_2} \right)$ will be close to the negative real axis, so $\theta$ will be close to $\pi$. In this case, event (b) in the above theorem is likely to occur, so the map $p$ given by $\lambda_j \mapsto \mu_j$ is the most likely eigenpairing. Conversely, if $v_2$ is large and $v_1$ small, then $q$ is the most likely eigenpairing. In this sense, we see that when an eigenvector of $A$ is close to an eigenvector of $B$, their corresponding eigenvalues are likely to pair.

In particular, if $\binom{1}{0}$ lies in the $\mu_1$-eigenspace of $B$, then $p$ is an eigenpairing due to Lemma \ref{SameEvecsLem}. Further, by Lemma \ref{TraceCondition} we may take the associated eigenpath set to be a pair of straight lines. In this case, $q$ is also an eigenpairing if and only if there is some $\alpha \in (0,1)$ such that $(1-\alpha) \lambda_1 + \alpha \mu_1 = (1-\alpha) \lambda_2 + \alpha \mu_2$, or equivalently $\frac{\mu}{\lambda} < 0$.

Similarly, if $\binom{1}{0}$ lies in the $\mu_2$-eigenspace of $B$, then $q$ is always an eigenpairing, and $p$ is also an eigenpairing if and only if $\frac{\mu}{\lambda} > 0$.
\end{Rmk}

\section{The Polynomial Analogs of Theorems 1-3}
A natural extension of the ideas considered so far is the variation in the roots of a complex polynomial of degree $n$ that varies continuously with a real parameter $\alpha$. In order to state the analogs of Theorems 1-3, we first provide some analogous definitions.

We again use $\spec(Q)$ to denote the size-$n$ multiset of roots of the polynomial $Q$, where $Q \in \C[t]$ has degree $n$. Additionally, we write $\|Q\| = \max_j |a_j|$ when $Q = \sum_{j=0}^n a_j t^j$.

\begin{Def}
A \emph{polynomial path} is a continuous function $[0,1] \to \C_n[t]$, where $\C_n[t]$ denotes the set of univariate complex monic polynomials of degree $n$. In this section, we will assume that $\PP_\alpha$ is a polynomial path with $\PP_0 = Q$ and $\PP_1 = R$. The $\emph{convex}$ polynomial path from $Q$ to $R$ is given by $P_\alpha = (1-\alpha)Q + \alpha R$.
\end{Def}

\begin{Def}
As before, the polynomial path $\PP_\alpha$ determines a \emph{$\PP$-root-region}
\begin{align*}
    E_\PP = \bigcup_{\alpha \in [0,1]} \{(\zeta, \alpha) \mid \PP_\alpha(\zeta) = 0\}
\end{align*}
and corresponding \emph{$\PP$-root-path sets} $\{\gamma_1, ..., \gamma_n\}$ such that
\begin{align*}
    \spec(\PP_\alpha) = \{\gamma_1(\alpha), ..., \gamma_n(\alpha)\}
\end{align*}
for all $\alpha \in [0,1]$. Such a $\PP$-root-path set determines a \emph{$\PP$-root-pairing} $\gamma_j(0) \mapsto \gamma_j(1)$.
\end{Def}

\begin{Def}
The \emph{companion matrix path} of $\PP_\alpha = t^n + \sum_{j=0}^{n-1} a_j(\alpha) t^j$ is the matrix path
\begin{align*}
    \CC_\alpha = \begin{bmatrix}
    0 & & & & -a_0(\alpha) \\
    1 & 0 & & & -a_1(\alpha) \\
     & 1 & \ddots & & \vdots \\
     & & \ddots & 0 & -a_{n-2}(\alpha) \\
     0 & & & 1 & -a_{n-1}(\alpha)
    \end{bmatrix},
\end{align*}
so the characteristic polynomial of each $\CC_\alpha$ is $\PP_\alpha$ (see \cite{MAnalysis}). This observation yields the following lemma.
\end{Def}

\begin{Lemma}
Let $\CC$ be the companion matrix path of $\PP$. Then $E_\CC = E_\PP$, and $\{\gamma_1, ..., \gamma_n\}$ is a $\CC$-eigenpath set if and only if it is a $\PP$-root-path set.
\end{Lemma}

We now state and prove the polynomial analogs of Theorems 1, 2, and 3.

\begin{Thm}[Analog of Theorem \ref{ClosePathsThm}]
Let $\varepsilon > 0$ and let $\PP$ be a polynomial path. Then there is a $\delta > 0$ such that for any polynomial path $\PP'$ with $\|\PP_\alpha - \PP_\alpha'\| < \delta$ for all $\alpha \in [0,1]$ and any $\PP'$-root-path set $\{\gamma_1', ..., \gamma_n'\}$, there is a $\PP$-root-path set $\{\gamma_1, ..., \gamma_n\}$ satisfying $|\gamma_j(\alpha) - \gamma_j'(\alpha)| < \varepsilon$ for all $\alpha \in [0,1]$.
\end{Thm}

\begin{proof}
Let $\CC$ be the companion matrix path of $\PP$. Using Theorem 1, we obtain $\delta>0$ such that for any matrix path $\CC'$ with $\|\CC_\alpha - \CC_\alpha'\| < \delta$ for all $\alpha \in [0,1]$ and any $\CC'$-eigenpath set $\{\gamma_1', ..., \gamma_n'\}$, there is a $\CC$-eigenpath set $\{\gamma_1, ..., \gamma_n\}$ satisfying $|\gamma_j(\alpha) - \gamma_j'(\alpha)| < \varepsilon$ for all $\alpha \in [0,1]$.

Let $\PP'$ be a polynomial path satisfying $\|\PP_\alpha - \PP_\alpha'\| < \delta$, and let $\{\gamma_1', ..., \gamma_n'\}$ be a $\PP'$-root-path set. Then if $\CC'$ is the companion matrix path of $\PP'$, we have $\|\CC_\alpha - \CC_\alpha'\| < \delta$ and a $\CC'$-eigenpath set $\{\gamma_1', ..., \gamma_n'\}$. Thus we may find a $\CC$-eigenpath set $\{\gamma_1, ..., \gamma_n\}$, such that $|\gamma_j(\alpha) - \gamma_j'(\alpha)| < \varepsilon$ for all $\alpha \in [0,1]$. Since $\{\gamma_1, ..., \gamma_n\}$ is also a $\PP$-root-path set, we are finished.
\end{proof}

\begin{Thm}[Analog of Theorem \ref{RippingThm}]
Let $\varepsilon > 0$ and let $\PP$ be a polynomial path with $\PP$-root-path set $\{\gamma_1, ..., \gamma_n\}$. Then there is a polynomial path $\PP'$ admitting a unique $\PP'$-root-path set $\{\gamma_1', ..., \gamma_n'\}$ such that $\|\PP_\alpha - \PP_\alpha'\| < \varepsilon$ and $|\gamma_j(\alpha) - \gamma_j'(\alpha)| < \varepsilon$ for all $\alpha \in [0,1]$ and $j=1,...,n$.
\end{Thm}

\begin{proof}
Let $\CC$ be the companion matrix path of $\PP$. By Theorem \ref{RippingThm}, we know for any $\varepsilon' > 0$ there is a matrix path $\CC'$ admitting a unique $\CC'$-eigenpath set $\{\gamma_1', ..., \gamma_n'\}$ such that $\|\CC_\alpha - \CC'_\alpha\| < \varepsilon'$ and $|\gamma_j(\alpha) - \gamma_j'(\alpha)| < \varepsilon$ for $\alpha \in [0,1]$ and $j=1,...,n$. Since the characteristic polynomial of a matrix varies continuously with its entries, we may take $\varepsilon'$ to be small enough so that the characteristic polynomial $\PP_\alpha'$ of $\CC_\alpha'$ satisfies $\|\PP_\alpha - \PP_\alpha'\| < \varepsilon$. 
\end{proof}

\begin{Thm}[Analog of Theorem \ref{ConvexReductionThm}]
Let $f$ and $g$ be continuous functions $[0,1] \to \R$ satisfying $f(0) = g(1) = 1$ and $f(1) = g(0) = 0$ so that $\PP_\alpha = f(\alpha)Q+g(\alpha)R$ is a polynomial path from $Q$ to $R$. If $(f \lor g)(\alpha) \geq 0$ for all $\alpha \in [0,1]$, then any convex root-pairing $p$ is also a $\PP$-root-pairing.
\end{Thm}

\begin{proof}
Let $\CC$ be the companion matrix path of $\PP$, so $\CC_\alpha = f(\alpha) A + g(\alpha) B$ where $A$ and $B$ are the companion matrices of $Q$ and $R$, respectively. Further, the matrix path $C_\alpha = (1-\alpha)A + \alpha B$ is the companion matrix path of $P_\alpha = (1-\alpha)A + \alpha B$. By Theorem \ref{ConvexReductionThm}, any convex eigenpairing $p$ of $A$ and $B$ is also a $\CC$-eigenpairing. Since $E_C = E_P$ and $E_\CC = E_\PP$, it follows that the convex eigenpairings are exactly the convex root-pairings, and similarly that the $\CC$-eigenpairings are exactly the $\PP$-root-pairings. Thus, our result holds.
\end{proof}

\pagebreak

        
        
        


        
        
        
        

\bibliographystyle{apa}
\bibliography{refs}

\end{document}